\documentclass[a4paper,12pt,reqno]{amsart}

\usepackage[T1]{fontenc}
\usepackage{lmodern}
\usepackage[utf8]{inputenc}

\usepackage[left=2.2cm,right=2.2cm,top=2.2cm,bottom=2.7cm]{geometry}
\usepackage{amsthm,amsmath,amsfonts,amssymb,mathtools,extarrows}
\usepackage{booktabs}
\usepackage[bbgreekl]{mathbbol}
\DeclareSymbolFontAlphabet{\mathbb}{AMSb}
\DeclareSymbolFontAlphabet{\mathbbl}{bbold}
\usepackage[numbers]{natbib}
\usepackage[
    colorlinks,
    hyperfootnotes=true,
    hyperindex,
    linkcolor=blue,
]{hyperref}
\usepackage{xcolor}

\allowdisplaybreaks

\numberwithin{equation}{section}

\theoremstyle{plain}
\newtheorem{theorem}{Theorem}[section]
\newtheorem{lemma}[theorem]{Lemma}

\newtheorem{proposition}[theorem]{Proposition}
\newtheorem{corollary}[theorem]{Corollary}

\theoremstyle{definition}

\theoremstyle{remark}
\newtheorem{remark}[theorem]{Remark}
\newtheorem{example}[theorem]{Example}

\newcommand{\Nzero}{\mathbb{N}_0}

\newcommand{\R}{\mathbb{R}}

\newcommand{\PP}{\mathbb{P}}
\newcommand{\Un}{\mathbbl{1}}

\newcommand{\E}{\mathbb{E}}

\newcommand{\ensemble}[1]{ \left\lbrace #1 \right\rbrace }

\newcommand{\norm}[1]{\left\Vert #1 \right\Vert}

\newcommand{\Unens}[1]{ \Un_{ \ensemble{#1} } }
\newcommand{\1}[1]{\Unens{#1}}

\newcommand{\Delt}{\Delta}

\newcommand{\dd}{\,\mathrm d}

\title[Stein solution factors via Mills ratios]{Stein Solution Factors via Mills Ratios:\\Affine Birth Rates and Symmetric Potential Distributions}
\author[P. Eichelsbacher]{Peter Eichelsbacher}
\address{Faculty of Mathematics, Ruhr University Bochum, Germany}
\email{peter.eichelsbacher@rub.de}

\begin{document}

\begin{abstract}
We study Stein solution factors for indicator test functions by a common Mills-ratio method in discrete and continuous settings. In the discrete case, a general birth--death formulation gives exact solution envelopes and recovers the first-increment theory of Brown and Xia. For binomial and Poisson targets, and for negative binomial targets with shape parameter $r\ge1$, the additional algebraic structure of affine birth rates yields improved uniform bounds for the Stein solution. In the continuous case, we consider symmetric densities proportional to $e^{-V}$. Under natural structural assumptions on the potential $V$, Mills ratios yield explicit finite bounds for the indicator Stein solution and show that the uniform derivative and drift Stein factors have the sharp value $1$. Under an additional one-crossing condition, the solution-factor optimization can be carried out exactly and gives the optimal value $1/(4p(0))$. The even-power targets arising in statistical mechanics, including the quartic critical Curie--Weiss law, provide the original examples, and the calculation extends to the wider Subbotin family. For $1<\beta<2$ in the Subbotin family, the exact envelope has two off-center maximizers characterized by a unique incomplete-gamma equation, while the two first-order factors remain equal to $1$. The results exhibit a common discrete--continuous mechanism behind improved and, in the continuous sharp regime, optimal Stein solution factors.
\end{abstract}

\maketitle

\section{Introduction}

Stein's method turns distributional approximation into analytic control of a solution of a characterizing equation.  The quality of the resulting probability bound is therefore tied to the quality of the corresponding \emph{Stein factors}.  In this paper we focus on indicator test functions and hence on two direct distributional metrics: total variation for integer-valued targets and Kolmogorov distance for continuous targets.  Our purpose is to show that, in both settings, Mills ratios isolate the same analytic object---an exact tail-to-mass or tail-to-density envelope---and that structural information on the target can then be used to improve the uniform bound for the Stein solution itself.

The classical Gaussian case provides the continuous benchmark.  If $Z$ is standard normal, the Stein equation for $h_z=\1{(-\infty,z]}$ is
\begin{equation}\label{eq:intro-normal-stein}
 f_z'(x)-x f_z(x)=\1{(-\infty,z]}(x)-\Phi(z).
\end{equation}
The sharp classical factors are
\begin{equation}\label{eq:intro-normal-factors}
 \sup_{z\in\R}\norm{f_z}_\infty=\frac{\sqrt{2\pi}}4,
 \qquad
 \sup_{z\in\R}\norm{f_z'}_\infty=1;
\end{equation}
see, for example, \cite{Stein1972,ChenGoldsteinShao2011}.  The discrete benchmark is Poisson.  For $Z\sim\operatorname{Po}(\lambda)$ and $A\subseteq\Nzero$,
\begin{equation}\label{eq:intro-poisson-stein}
 \lambda g_A(k+1)-k g_A(k)=\1{A}(k)-\PP(Z\in A).
\end{equation}
For a function $h$ on the integers, we write $\Delt h(k):=h(k+1)-h(k)$.
Frequently used classical estimates are
\begin{equation}\label{eq:intro-poisson-factors}
 \norm{g_A}_\infty\le 1\wedge\sqrt{\frac{2}{e\lambda}},
 \qquad
 \norm{\Delt g_A}_\infty\le\frac{1-e^{-\lambda}}{\lambda}\le 1\wedge\frac1\lambda;
\end{equation}
see \cite{BarbourHolstJanson1992,BrownXia2001}.  One of the elementary consequences of the affine calculation below is the stronger solution bound $1\wedge(2\lambda)^{-1/2}$, and keeping the full maximizer information yields a further finite-parameter improvement.

The continuous and discrete calculations are governed by the same structure.  For a continuous target
\[
 p(x)=Z_V^{-1}e^{-V(x)},
 \qquad
 F(x):=\int_{-\infty}^{x}p(t)\,\dd t,
 \qquad
 g(x)=V'(x)=-\frac{p'(x)}{p(x)},
\]
the Stein operator is $f'-gf$, and the two Mills ratios are $F/p$ and $(1-F)/p$.  For a discrete target $\pi=(\pi_k)_{k=0}^{N}$ on $\{0,1,\ldots,N\}$, where $N\in\mathbb N\cup\{\infty\}$, write
\[
 F_{k-1}:=\sum_{j=0}^{k-1}\pi_j,
 \qquad
 Q_k:=\sum_{j=k}^{N}\pi_j,
 \qquad k\ge1.
\]
In the normalization with death rate $k$, the corresponding local quantity is the birth rate
\[
 b_k=(k+1)\frac{\pi_{k+1}}{\pi_k},
\]
and the two Stein-adapted Mills ratios are
\[
 \frac{F_{k-1}}{k\pi_k},
 \qquad
 \frac{Q_k}{k\pi_k}.
\]
Most importantly, on the symmetric convex-unimodal continuous class used below, and without any extra restriction in the discrete case, the indicator solution has an exact one-dimensional envelope:
\begin{equation}\label{eq:intro-parallel-envelopes}
 \sup_A|g_A(k)|=\frac{F_{k-1}Q_k}{k\pi_k}
 \qquad\longleftrightarrow\qquad
 \norm{f_z}_\infty=\frac{F(z)(1-F(z))}{p(z)}.
\end{equation}
Thus the main analytic problem is not the formal solution of the Stein equation, but the optimization of the explicit envelope in \eqref{eq:intro-parallel-envelopes}.

The discrete part is treated first.  Brown and Xia's birth--death theory already gives a broad and sharp first-increment theory.  In our normalization their condition (C2) is exactly the statement that the two discrete Mills ratios have opposite monotonicities, and their simpler sufficient condition (C4) is a one-step condition on the birth rates.  We recover their pointwise increment factor in a few lines.  Our new question concerns the global optimization of the zeroth-order Stein factor.  Starting from the classical exact tail envelope, we restrict to affine birth rates
\[
 b_k=a+\rho k.
\]
This restriction is not merely a convenience: the stationary affine family consists exactly of the binomial ($\rho<0$), Poisson ($\rho=0$), and negative binomial ($0<\rho<1$) laws.  These three distributions are also canonical quantitative approximation targets in the Stein--Chen literature.  Chen's foundational dependent-trials argument initiated the modern Poisson approximation method, and Barbour and Holst developed it further in a range of Poisson convergence problems \cite{Chen1975,BarbourHolst1989}.  For binomial approximation, Ehm's Stein--Chen treatment of Poisson-binomial sums and R\"ollin's centered binomial approximation for locally dependent variables are representative examples \cite{Ehm1991,Rollin2008}.  Brown and Phillips developed negative binomial approximation by Stein's method and, in particular, quantified convergence to a negative binomial limit for the P\'olya distribution \cite{BrownPhillips1999}.  The generator approach of Eichelsbacher and Reinert places Poisson, binomial and negative binomial distributions in the broader class of discrete Gibbs measures \cite{EichelsbacherReinert2008}.  Choi \cite{Choi2018} subsequently bounded discrete Stein solutions through hitting and mixing times; for the binomial example this gives a logarithmic-order uniform solution bound in the support size, illustrating both the reach and the possible looseness of a general Markov-chain approach.

The novelty below lies not in the tail representation itself, but in its global optimization under the affine birth-rate structure.  Affinity supplies the rigidity needed to compare the exact solution envelope at a maximizing index with both neighboring lattice values.  This produces a simple parameter bound for the solution itself and, when the location of the maximizing lattice point is retained, a sharper finite-parameter correction.  For negative binomial targets with $r\ge1$ and for binomial targets these bounds improve the explicit arbitrary-indicator solution estimates located in the existing literature; for Poisson the same calculation gives a simple improvement of the standard uniform solution bound.

The potential formulation $p\propto e^{-V}$ also has a substantial Stein-method history beyond the Gaussian potential $V(x)=x^2/2$ and quartic potentials $V(x)=\lambda x^4$.  Chatterjee and Shao formulate non-normal approximation for densities of the form $p(x)\propto e^{-c_0G(x)}$ arising from a nonlinear exchangeable-pair regression, and explicitly treat the power-exponential family $p(x)\propto e^{-|x|^\alpha/\beta}$ \cite{ChatterjeeShao2011}.  Eichelsbacher and L\"owe obtain quantitative limits in generalized Curie--Weiss models with densities proportional to
\[
 \exp\left\{-\frac{\mu |x|^{2k}}{(2k)!}\right\},
 \qquad k\ge1,
\]
including higher-order critical examples beyond the quartic case \cite{EichelsbacherLoewe2010}.  Shao and Zhang subsequently developed Berry--Esseen bounds for a broad non-normal exchangeable-pair framework and applied it to general Curie--Weiss and related models \cite{ShaoZhang2019}.  These results provide a natural precedent for treating $V$ structurally rather than restricting attention to a single polynomial potential.

During the final preparation of this manuscript, Rapin and Swan \cite{RapinSwan2026} posted a substantially broader kernel calculus for first-order Stein equations and their derivatives.  Their treatment of the Subbotin (or exponential-power) family
\[
 p_\beta(x)=C_\beta
 \exp\left\{-\frac{|x|^\beta}{\beta(\beta-1)}\right\},
 \qquad \beta>1,
\]
prompted us to ask whether the even-power calculation already underlying the continuous examples of the present manuscript extends, within the elementary Mills-ratio and one-crossing framework, to non-integer exponents.  It does on the range $\beta\ge2$, with the same sharp Kolmogorov solution and first-derivative factors on the overlap.  The remaining interval $1<\beta<2$ falls outside our smoothness assumptions at the origin and is treated separately in Proposition~\ref{prop:subbotin-below-two}.  The chronology and the detailed comparison are recorded in Section~\ref{sec:continuous-examples}.

The continuous part of the present paper was originally motivated by the quartic critical Curie--Weiss density
\begin{equation}\label{eq:intro-quartic}
 p_\lambda(x)=\frac{2\lambda^{1/4}}{\Gamma(1/4)}e^{-\lambda x^4},\qquad \lambda>0,
\end{equation}
with the critical Curie--Weiss normalization $\lambda=1/12$.  The quartic law goes back to the classical critical limit theory of Ellis and Newman \cite{EllisNewman1978,EllisNewmanRosen1980}; quantitative non-normal Stein approximation was developed in \cite{EichelsbacherLoewe2010,ChatterjeeShao2011,ShaoZhang2019}, with subsequent non-normal moderate-deviation and non-uniform refinements \cite{ShaoZhangZhang2021,ThanhTu2025}.  Related quartic limits arise in self-organized criticality \cite{CerfGornySOC,GornyGaussian2014}, Curie--Weiss--Potts models \cite{GandolfoRuizWouts2010,EichelsbacherMartschink2015}, and the cubic mean-field Ising model \cite{ContucciMingioneOsabutey2024,EichelsbacherCubic2024}.  For the critical Ising model on sparse random graphs, Prodromidis and Sly \cite{ProdromidisSly2026} obtain a deterministic quartic limiting density in the random-regular case, whereas the Erd\H{o}s--R\'enyi case has a non-deterministic limiting law.  Rather than treating \eqref{eq:intro-quartic} first, we begin directly with the general symmetric density $p=Z_V^{-1}e^{-V}$ and return at the end to the even-power targets and their Subbotin extension.

Our continuous assumptions are deliberately hierarchical.  A basic convex-unimodal condition already yields the upper Mills bound, the inequalities $\|f_z'\|_\infty\le1$ and $\|V'f_z\|_\infty\le1$, and a finite explicit solution bound.  A second condition, equivalent to convexity of the reciprocal drift $1/V'$, yields a two-sided Mills estimate, the tail asymptotic $V'(x)(1-F(x))/p(x)\to1$, sharpness of the two constants $1$, and monotonicity of $V'f_z$.  To identify the optimal solution factor $1/(4p(0))$ one needs additional global shape information.  We formulate this as a one-crossing criterion, but also record the exact weaker condition: the one-crossing property is sufficient, not necessary.  A central point is that these are sharpness and optimization statements, not merely improved upper bounds: in the non-normal Stein papers used for comparison below, usable solution estimates are derived for approximation purposes; Eichelsbacher and L\"owe explicitly note that their solution constant is not optimal, while the exact uniform solution factor and the sharpness of the first-order constants are not determined in those comparison results.

The rest of the paper is organized accordingly.  Sections~\ref{sec:discrete-general}--\ref{sec:binomial} contain the discrete theory, with the literature comparison placed next to each family-specific result.  Sections~\ref{sec:continuous-general}--\ref{sec:continuous-examples} develop the continuous theory and its examples.

\section{Discrete Mills ratios and the Brown--Xia increment structure}\label{sec:discrete-general}

Let $\pi=(\pi_k)$ be a probability law on $\{0,1,\dots,N\}$, where $N\in\mathbb N\cup\{\infty\}$, with positive mass on its support.  We use the unit-death normalization
\begin{equation}\label{eq:general-operator}
 \mathcal A g(k)=b_k g(k+1)-kg(k),
\end{equation}
where detailed balance is
\begin{equation}\label{eq:detailed-balance}
 b_k\pi_k=(k+1)\pi_{k+1}.
\end{equation}
For finite support we put $b_N=0$.  Write
\[
 F_{k-1}:=\sum_{j=0}^{k-1}\pi_j,
 \qquad
 Q_k:=\sum_{j=k}^{N}\pi_j.
\]
For $k\ge1$ define the two discrete Mills ratios
\begin{equation}\label{eq:general-mills}
 M_k^-:=\frac{F_{k-1}}{k\pi_k},
 \qquad
 M_k^+:=\frac{Q_k}{k\pi_k}.
\end{equation}
Detailed balance gives the recurrences
\begin{equation}\label{eq:general-rec}
 b_kM_{k+1}^- - kM_k^-=1,
 \qquad
 b_kM_{k+1}^+ - kM_k^+=-1.
\end{equation}
For $h_A=\1{A}$, the bounded solution normalized by $g_A(0)=0$ is
\begin{equation}\label{eq:general-solution-mills}
 g_A(k)
 =M_k^+\PP(Z\in A,\,Z<k)
 -M_k^-\PP(Z\in A,\,Z\ge k).
\end{equation}
Consequently,
\begin{equation}\label{eq:general-envelope}
 \sup_A|g_A(k)|
 =H_k:=\frac{F_{k-1}Q_k}{k\pi_k}
 =\frac{M_k^-M_k^+}{M_k^-+M_k^+}.
\end{equation}
For fixed $k$, the solution value $g_h(k)$ depends linearly on the coordinates $h(j)\in[0,1]$.  Hence the maximum of $|g_h(k)|$ over the product cube is attained at an extreme point, that is, at an indicator function.  Thus the same envelope holds for the full class $0\le h\le1$, and we write
\begin{equation}\label{eq:C0-general}
 C_0(\pi):=\sup_{0\le h\le1}\norm{g_h}_\infty
 =\sup_{k\ge1}H_k.
\end{equation}
The exact envelope \eqref{eq:general-envelope} is already contained, in passage-time language, in the general birth--death theory of Brown and Xia \cite{BrownXia2001}; we use it as the starting point for the optimization problem.

\begin{proposition}[Mills form of the Brown--Xia condition]\label{prop:BX-mills}
For every interior $k$, the following are equivalent:
\begin{equation}\label{eq:mills-signs}
 \Delt M_k^-\ge0,
 \qquad
 \Delt M_k^+\le0,
\end{equation}
and
\begin{equation}\label{eq:C2}
 \frac{F_k}{F_{k-1}}
 \ge \frac{b_k}{k}
 \ge \frac{Q_{k+1}}{Q_k}.
\end{equation}
Under these equivalent conditions,
\begin{equation}\label{eq:increment-pointwise}
 \sup_A|\Delt g_A(k)|
 =\frac{F_{k-1}}{k}+\frac{Q_{k+1}}{b_k}.
\end{equation}
\end{proposition}

\begin{proof}
By detailed balance,
\[
 M_{k+1}^-=\frac{F_k}{b_k\pi_k},
 \qquad
 M_{k+1}^+=\frac{Q_{k+1}}{b_k\pi_k},
\]
so \eqref{eq:mills-signs} is exactly \eqref{eq:C2}.  Subtracting \eqref{eq:general-solution-mills} at $k$ and $k+1$ gives
\begin{align*}
 \Delt g_A(k)
 ={}&\Delt M_k^+\PP(Z\in A,Z<k)
      -\Delt M_k^-\PP(Z\in A,Z>k)\\
 &+\1{A}(k)\pi_k\{M_k^-+M_{k+1}^+\}.
\end{align*}
Under \eqref{eq:mills-signs}, the first two coefficients are nonpositive and the last is positive.  Their total sum is zero, because $g_{\operatorname{supp}(\pi)}\equiv0$.  Hence the largest absolute value is the positive coefficient, which equals
\[
 \pi_k\{M_k^-+M_{k+1}^+\}
 =\frac{F_{k-1}}{k}+\frac{Q_{k+1}}{b_k}.
\]
\end{proof}

\begin{remark}[Fair comparison with Brown--Xia]\label{rem:BX}
Condition \eqref{eq:C2} is precisely condition (C2) of Brown and Xia \cite{BrownXia2001} in the unit-death normalization.  Their corresponding passage-time monotonicity (C1) is exactly \eqref{eq:mills-signs}; their conditions (C0)--(C3) are equivalent, while the simple sufficient condition (C4) becomes
\begin{equation}\label{eq:C4}
 b_k-b_{k-1}\le1.
\end{equation}
Thus \eqref{eq:increment-pointwise} is Brown--Xia's pointwise first-increment factor written in Mills-ratio language.  The argument is short and analytic, but it does not improve their increment theory.  The related discrete-Gibbs formulation of Eichelsbacher and Reinert \cite{EichelsbacherReinert2008} makes the same connection with Brown--Xia's tail-ratio condition.
\end{remark}

\begin{remark}[Potential notation]
If $\pi_k\propto e^{-V(k)}$, then
\[
 b_k=(k+1)e^{V(k)-V(k+1)}.
\]
Hence \eqref{eq:C2} and \eqref{eq:C4} can be rewritten as conditions on finite differences of a discrete potential.  We shall not use this parametrization below; the birth rate is the more direct quantity for the affine calculation.
\end{remark}

\begin{remark}[Three classical increment factors]\label{rem:three-increments}
For the three affine families considered below, condition \eqref{eq:C4} is automatic and Proposition~\ref{prop:BX-mills} specializes as follows.  For Poisson, $b_k=\lambda$ and
\[
 \sup_A|\Delt g_A(k)|=\frac{F_{k-1}}k+\frac{Q_{k+1}}\lambda,
 \qquad
 \sup_A\norm{\Delt g_A}_\infty=\frac{1-e^{-\lambda}}\lambda,
\]
the classical Poisson increment factor; see, for example, \cite{BarbourHolstJanson1992}.  For $\operatorname{NB}(r,p)$, with $q=1-p$ and $b_k=q(r+k)$,
\[
 \sup_A|\Delt g_A(k)|=\frac{F_{k-1}}k+\frac{Q_{k+1}}{q(r+k)},
 \qquad
 \sup_A\norm{\Delt g_A}_\infty=\frac{1-p^r}{rq},
\]
the corresponding classical negative-binomial increment factor; see \cite{BrownPhillips1999}.  For $\operatorname{Bin}(n,p)$, in the usual normalization
\[
 p(n-k)f_A(k+1)-qk f_A(k)=\1{A}(k)-\PP(Z\in A),
\]
one obtains, for $1\le k\le n-1$,
\[
 \sup_A|\Delt f_A(k)|=\frac{F_{k-1}}{qk}+\frac{Q_{k+1}}{p(n-k)},
\]
and
\[
 \sup_A\norm{\Delt f_A}_\infty
 \le \frac{1-p^{n+1}-q^{n+1}}{(n+1)pq},
\]
the classical binomial increment bound; see \cite{Ehm1991}.
Thus the first-increment side of the affine picture is classical.  The role of affinity below is different: it makes the global solution-factor optimization algebraically rigid.
\end{remark}

\section{Affine birth rates: bounds for the solution envelope}\label{sec:affine-principle}

The representation \eqref{eq:general-envelope} reduces the zeroth-order Stein-factor problem to a one-dimensional optimization: by \eqref{eq:C0-general}, it is enough to control $H_k$ at an index where it is maximal.  For general birth rates, the two comparisons $H_k\ge H_{k-1}$ and $H_k\ge H_{k+1}$ involve the unrelated neighboring rates $b_{k-1}$ and $b_k$.  For affine birth rates,
\begin{equation}\label{eq:affine-birth}
 b_k=a+\rho k,
 \qquad a>0,
 \qquad \rho<1,
\end{equation}
the constant increment $b_k-b_{k-1}=\rho$ makes these two comparisons compatible.  The stationary affine cases are exactly
\begin{center}
\begin{tabular}{@{}lll@{}}
\toprule
Slope & Target & Parameters \\
\midrule
$\rho<0$ & binomial & $a=-\rho n$, finite support $\{0,\dots,n\}$ \\
$\rho=0$ & Poisson & $a=\lambda$ \\
$0<\rho<1$ & negative binomial & $\rho=q$, $a=qr$ \\
\bottomrule
\end{tabular}
\end{center}
Indeed, detailed balance identifies the positive-slope case with $\operatorname{NB}(a/\rho,1-\rho)$, while a negative affine slope forces a finite endpoint and hence a binomial law.

The theorem below records the resulting bounds on the solution envelope itself.  The neighboring maximality inequalities and the auxiliary algebra used to obtain them are kept in the proof.

\begin{theorem}[Affine solution-factor bound]\label{thm:affine-max}
Assume \eqref{eq:affine-birth}, let $a\ge\rho$, and let $k$ be a maximizing index of $H_j$ for which $b_k>0$.

\emph{(a)} If $a>\rho$, then
\begin{equation}\label{eq:coarse-H}
 H_k\le\frac1{\sqrt{2(a-\rho)}}.
\end{equation}

\emph{(b)} More precisely,
\begin{equation}\label{eq:refined-affine}
 H_k
 \le
 \left[
 2(a-\rho)
 +\bigl\{(1-\rho)k-(a-\rho)\bigr\}^2
 +\left\{\frac{a-\rho+2\rho k}
 {a-\rho+(1+\rho)k}\right\}^2
 \right]^{-1/2}.
\end{equation}
\end{theorem}

\begin{proof}
Write, only for the proof,
\[
 u:=(M_k^-)^{-1},
 \qquad
 v:=(M_k^+)^{-1},
 \qquad
 c:=b_{k-1}=a+\rho(k-1).
\]
By \eqref{eq:general-envelope},
\[
 H_k=\frac1{u+v}.
\]
The Mills recurrences give
\[
 M_{k+1}^-=\frac{kM_k^-+1}{b_k},
 \qquad
 M_{k+1}^+=\frac{kM_k^+-1}{b_k}.
\]
Thus $H_k\ge H_{k+1}$ is equivalent to
\begin{equation}\label{eq:neighbor-right}
 (k+u)(k-v)\le kb_k.
\end{equation}
For $k\ge2$, the backward recurrences give
\[
 M_{k-1}^-=\frac{cM_k^- -1}{k-1},
 \qquad
 M_{k-1}^+=\frac{cM_k^+ +1}{k-1},
\]
and $H_k\ge H_{k-1}$ is equivalent to
\begin{equation}\label{eq:neighbor-left}
 (c-u)(c+v)\le c(k-1).
\end{equation}
If $k=1$, then $M_1^-=1/b_0$, hence $u=c$, and \eqref{eq:neighbor-left} again holds, now with equality.

Put
\[
 D:=v-u,
 \qquad
 P:=uv.
\]
Expanding \eqref{eq:neighbor-right}--\eqref{eq:neighbor-left} yields
\begin{equation}\label{eq:two-P-bounds}
 P\ge k\{k-c-\rho-D\},
 \qquad
 P\ge c\{1-k+c+D\}.
\end{equation}
Adding the two inequalities and using $(u+v)^2=D^2+4P$ gives
\begin{equation}\label{eq:coarse-square}
 (u+v)^2
 \ge \{D-(k-c)\}^2+(k-c)^2+2(a-\rho).
\end{equation}
Dropping the two squares proves part~(a).

For the sharper bound, retain the two inequalities in \eqref{eq:two-P-bounds} separately.  Their affine right-hand sides intersect at
\[
 D_*=k-c-\frac{\rho k+c}{k+c},
\]
with common value $kc(1-\rho)/(k+c)$.  Since $b_k=c+\rho>0$ and $\rho<1$, one checks that $D_*<2k$ and $D_*>-2c$.  If $\phi$ denotes the displayed piecewise quadratic lower bound, then its one-sided derivatives at the intersection satisfy
\[
 \phi'_-(D_*)=2D_*-4k<0,
 \qquad
 \phi'_+(D_*)=2D_*+4c>0.
\]
Hence
\[
 D^2+4\max\bigl\{k(k-c-\rho-D),\,c(1-k+c+D)\bigr\}
\]
is minimized at $D_*$.  Consequently,
\begin{align*}
 H_k^{-2}=(u+v)^2
 &\ge 2(a-\rho)+(k-c)^2
 +\left(\frac{\rho k+c}{k+c}\right)^2\\
 &=2(a-\rho)
 +\bigl\{(1-\rho)k-(a-\rho)\bigr\}^2
 +\left\{\frac{a-\rho+2\rho k}
 {a-\rho+(1+\rho)k}\right\}^2.
\end{align*}
Since $a\ge\rho$, the right-hand side is positive, and part~(b) follows.
\end{proof}

\begin{remark}[Immediate Poisson consequence]\label{rem:poisson-coarse}
For Poisson, $a=\lambda$ and $\rho=0$, so part~(a) gives $(2\lambda)^{-1/2}$ at a maximizing index.  Combined with the elementary bound $C_0^{\rm P}(\lambda)\le1$, this yields
\[
 C_0^{\rm P}(\lambda)\le 1\wedge(2\lambda)^{-1/2}.
\]
The lattice-sensitive improvement from part~(b) is used below.
\end{remark}

\section{Nonnegative affine slopes: Poisson and negative binomial}\label{sec:nonnegative-slopes}

We now apply Theorem~\ref{thm:affine-max} to the two infinite-support affine families.  The zero-slope case is Poisson, whereas $0<\rho<1$ gives the negative binomial family.  In both cases the supremum in \eqref{eq:C0-general} is attained.  Indeed, detailed balance gives
\[
 \frac{\pi_{k+1}}{\pi_k}=\frac{a+\rho k}{k+1}\longrightarrow\rho<1.
\]
Hence $Q_k/\pi_k$ is bounded for all sufficiently large $k$ by a geometric-tail comparison, and therefore
\[
 0\le H_k\le\frac{Q_k}{k\pi_k}\longrightarrow0.
\]
Thus a maximizing index exists.  A common one-dimensional correction is available before the two specializations.

\par\addvspace{1.2\baselineskip}

\begin{lemma}[Uniform correction for nonnegative slopes]\label{lem:kappa-rho}
Let $0\le\rho<1$ and $a\ge\rho$.  Define
\begin{equation}\label{eq:kappa-rho}
 \kappa(\rho)
 :=\min_{0\le e\le1-\rho}
 \left\{
 e^2+\left(\frac{1+\rho-e}{2-e}\right)^2
 \right\}.
\end{equation}
At every maximizing index to which Theorem~\ref{thm:affine-max} applies,
\begin{equation}\label{eq:kappa-control-general}
 \bigl\{(1-\rho)k-(a-\rho)\bigr\}^2
 +\left\{\frac{a-\rho+2\rho k}
 {a-\rho+(1+\rho)k}\right\}^2
 \ge\kappa(\rho).
\end{equation}
Moreover, $\kappa(\rho)>\rho$ for $0<\rho<1$, while $\kappa(0)>0$.
\end{lemma}

\begin{proof}
Put, only in this proof,
\[
 e:=(1-\rho)k-(a-\rho).
\]
Then $e\le(1-\rho)k$ and the second term on the left-hand side of \eqref{eq:kappa-control-general} is
\[
 \left(\frac{(1+\rho)k-e}{2k-e}\right)^2.
\]
If $e\le0$, this ratio is at least $(1+\rho)/2$, and the claim follows by comparing with the value $e=0$ in \eqref{eq:kappa-rho}.  Suppose $e>0$.  For fixed $e$ the ratio is increasing in $k$, since
\[
 \frac{\partial}{\partial k}
 \frac{(1+\rho)k-e}{2k-e}
 =\frac{(1-\rho)e}{(2k-e)^2}>0.
\]
If $0<e\le1-\rho$, then $k\ge1$ gives
\[
 \frac{(1+\rho)k-e}{2k-e}
 \ge\frac{1+\rho-e}{2-e},
\]
and \eqref{eq:kappa-control-general} follows directly from \eqref{eq:kappa-rho}.  If $e>1-\rho$, the constraint $e\le(1-\rho)k$ gives $k\ge e/(1-\rho)$ and hence
\[
 \frac{(1+\rho)k-e}{2k-e}
 \ge\frac{2\rho}{1+\rho}.
\]
The resulting lower bound is at least the value of the function in \eqref{eq:kappa-rho} at the endpoint $e=1-\rho$, and therefore again at least $\kappa(\rho)$.

Finally, with $x=1-e$,
\[
 e^2+\left(\frac{1+\rho-e}{2-e}\right)^2-\rho
 =\frac{\left(\rho-\frac{1+x^2}{2}\right)^2+\frac34(1-x^2)^2}{(1+x)^2},
\]
which is strictly positive on the relevant range for $0<\rho<1$, and positive also for $\rho=0$.
\end{proof}

\par\addvspace{1.8\baselineskip}

\subsection{Poisson: the zero-slope case}\label{subsec:poisson}

Let $Z\sim\operatorname{Po}(\lambda)$.  The exact solution factor is
\[
 C_0^{\rm P}(\lambda)
 =\sup_{k\ge1}\frac{F_{k-1}Q_k}{kp_k}.
\]
The elementary estimate $C_0^{\rm P}(\lambda)\le1$ follows from
\[
 kp_k=\operatorname{Cov}(Z,\1{Z\ge k})
 =F_{k-1}Q_k\bigl(\E[Z\mid Z\ge k]-\E[Z\mid Z\le k-1]\bigr)
 \ge F_{k-1}Q_k.
\]

\par\addvspace{1.2\baselineskip}

\begin{corollary}[Poisson solution factor]\label{cor:poisson-factor}
For every $\lambda>0$,
\begin{equation}\label{eq:poisson-lattice}
 C_0^{\rm P}(\lambda)
 \le
 \frac1{\sqrt{2\lambda+\kappa_{\rm P}(\lambda)}},
\end{equation}
where
\begin{equation}\label{eq:kappa-P}
 \kappa_{\rm P}(\lambda)
 :=\min_{k\ge1}
 \left\{(k-\lambda)^2+
 \left(\frac{\lambda}{k+\lambda}\right)^2\right\}.
\end{equation}
In particular,
\begin{equation}\label{eq:poisson-cstar}
 C_0^{\rm P}(\lambda)
 \le
 \min\left\{1,\frac1{\sqrt{2\lambda+c_*}}\right\},
\end{equation}
where
\begin{equation}\label{eq:cstar}
 c_*:=\kappa(0)
 =\min_{0\le x\le1}\left\{(1-x)^2+\left(\frac{x}{1+x}\right)^2\right\}
 =0.2333389940\ldots .
\end{equation}
\end{corollary}

\begin{proof}
For $a=\lambda$ and $\rho=0$, Theorem~\ref{thm:affine-max}(b) gives, at a maximizing index,
\[
 H_k^{-2}\ge2\lambda+(k-\lambda)^2+\left(\frac{\lambda}{k+\lambda}\right)^2,
\]
which yields \eqref{eq:poisson-lattice}.  The bound \eqref{eq:poisson-cstar} follows from Lemma~\ref{lem:kappa-rho} and $C_0^{\rm P}(\lambda)\le1$.
\end{proof}

\par\addvspace{1.2\baselineskip}

\begin{remark}[Positioning of the Poisson bound]
The classical bound in \eqref{eq:intro-poisson-factors} has leading constant $\sqrt{2/e}$.  The basic affine estimate in Remark~\ref{rem:poisson-coarse} improves this to $1/\sqrt2$, and Corollary~\ref{cor:poisson-factor} adds a strictly positive finite-parameter correction.  We do not make a separate novelty claim for the latter refinement here.  Its leading constant remains $1/\sqrt2$.  For comparison, the exact envelope has the smaller central asymptotic constant $\sqrt{\pi/8}$: if $k=k_\lambda$ is a central lattice point with $k_\lambda-\lambda=O(1)$, then the central and local central limit theorems give
\[
 F_{k_\lambda-1}\longrightarrow\frac12,\qquad
 Q_{k_\lambda}\longrightarrow\frac12,\qquad
 p_{k_\lambda}\sim(2\pi\lambda)^{-1/2},\qquad
 k_\lambda\sim\lambda.
\]
Hence \eqref{eq:general-envelope} yields
\[
 H_{k_\lambda}
 \sim\frac{1/4}{\lambda(2\pi\lambda)^{-1/2}}
 =\sqrt{\frac{\pi}{8}}\,\lambda^{-1/2}.
\]
This is only an asymptotic benchmark for the exact solution envelope and is not used in the proof of Corollary~\ref{cor:poisson-factor}.
\end{remark}

\par\addvspace{1.8\baselineskip}

\subsection{Negative binomial: positive affine slope}\label{subsec:NB}

We next turn to $Z\sim\operatorname{NB}(r,p)$, with $q=1-p$ and
\begin{equation}\label{eq:NB-pmf}
 \pi_k=\frac{\Gamma(r+k)}{\Gamma(r)k!}p^r q^k.
\end{equation}
Here $a=qr$ and $\rho=q$.  The affine solution-factor bound gives the $r^{-1/2}$ part of the desired estimate.  To retain a useful bound also outside the large-$r$ regime, we first prove the independent elementary estimate $C_0^{\rm NB}\le1$ for $r\ge1$.

\par\addvspace{1.2\baselineskip}

\begin{lemma}[Elementary negative-binomial envelope bound]\label{lem:NB-one}
Let $Z\sim\operatorname{NB}(r,p)$ with $r\ge1$, $q=1-p$, and mass function $\pi_k$ as in \eqref{eq:NB-pmf}.  Then, for every $k\ge1$,
\begin{equation}\label{eq:NB-pointwise-one}
 F_{k-1}Q_k\le k\pi_k.
\end{equation}
Consequently,
\begin{equation}\label{eq:NB-one}
 C_0^{\rm NB}(r,p)\le1.
\end{equation}
For $r=1$ equality holds in \eqref{eq:NB-one}.
\end{lemma}

\begin{proof}
Fix $k\ge1$, and let $G$ and $\beta$ denote the cdf and density of a $\operatorname{Beta}(r,k)$ law.  The beta--negative-binomial identities give
\[
 F_{k-1}=G(p),\qquad
 Q_k=1-G(p),\qquad
 k\pi_k=pq\,\beta(p).
\]
Thus it remains to prove
\[
 G(p)\{1-G(p)\}\le pq\,\beta(p).
\]
We verify this directly for the parameters $r$ and $k$ needed here.  Put, for $0\le x\le1$,
\[
 J(x):=x(1-x)\beta(x)-G(x)\{1-G(x)\}.
\]
A direct differentiation gives
\[
 J'(x)=\beta(x)K(x),
 \qquad
 K(x):=r-1-(r+k)x+2G(x),
\]
and $K'(x)=2\beta(x)-(r+k)$.  If $J$ had a negative interior minimum at $x_0$, then $K(x_0)=0$ and $K'(x_0)\ge0$, so $\beta(x_0)\ge(r+k)/2$.  Put
\[
 z:=(r+k)x_0-r.
\]
From $K(x_0)=0$ we obtain
\[
 G(x_0)=\frac{1+z}{2},\qquad -1\le z\le1.
\]
Moreover,
\[
 2(r+k)x_0(1-x_0)-(1-z^2)
 =\frac{2(r-1)(k-1)+(r-1)(z-1)^2+(k-1)(z+1)^2}{r+k}\ge0.
\]
Hence
\[
 x_0(1-x_0)\beta(x_0)
 \ge\frac{1-z^2}{4}
 =G(x_0)\{1-G(x_0)\},
\]
a contradiction.  Therefore $J\ge0$, and evaluating at $x=p$ proves \eqref{eq:NB-pointwise-one} and hence \eqref{eq:NB-one}.

If $r=1$, then
\[
 H_k=\frac{1-q^k}{kp}\le1,
\]
with equality at $k=1$.
\end{proof}

\par\addvspace{1.2\baselineskip}

\begin{corollary}[Negative binomial solution factor]\label{cor:NB-factor}
Let $r\ge1$.  Then
\begin{equation}\label{eq:NB-kappa}
 C_0^{\rm NB}(r,p)
 \le
 \min\left\{1,
 \frac1{\sqrt{2q(r-1)+\kappa(q)}}\right\},
\end{equation}
and, in particular,
\begin{equation}\label{eq:NB-simple}
 C_0^{\rm NB}(r,p)
 \le
 \min\left\{1,\frac1{\sqrt{q(2r-1)}}\right\}.
\end{equation}
For $r=1$, $C_0^{\rm NB}(1,p)=1$ exactly.
\end{corollary}

\begin{proof}
Apply Lemma~\ref{lem:kappa-rho} with $a=qr$ and $\rho=q$, and combine the resulting affine bound with Lemma~\ref{lem:NB-one}.  Since $\kappa(q)>q$, \eqref{eq:NB-simple} follows from \eqref{eq:NB-kappa}.
\end{proof}

\par\addvspace{1.2\baselineskip}

\begin{remark}[Comparison with previous negative-binomial solution factors]
The exact pointwise envelope \eqref{eq:general-envelope} is already implicit in the passage-time representation of Brown and Xia \cite{BrownXia2001} and is not claimed as new.  Brown and Phillips \cite{BrownPhillips1999} obtained earlier explicit bounded-test estimates.  Translating Proposition~4.5 of Cloez and Delplancke \cite{CloezDelplancke2019} into our parametrization gives
\begin{equation}\label{eq:CD}
 C_0^{\rm NB}(r,p)
 \le
 \min\left\{\frac1p,
 \frac{\sqrt\pi}{\sqrt{(r+1)pq}}\right\}.
\end{equation}
For $r\ge1$, \eqref{eq:NB-simple} is nowhere worse than \eqref{eq:CD}, and \eqref{eq:NB-kappa} is stronger still.  To the best of our knowledge these improved uniform bounded-test factors are new; no claim of optimality is made.  General solution-factor frameworks for discrete distributions include \cite{EichelsbacherReinert2008,ErnstSwan2022}.
\end{remark}

\section{Negative affine slopes: the binomial case}\label{sec:binomial}

We now turn to the remaining affine regime $\rho<0$.  Nonnegativity of the birth rates forces a finite upper endpoint, and detailed balance leads to the binomial family.  The affine solution-factor bound remains valid, but there is one additional issue: the maximum of the exact envelope may occur at the upper endpoint, where no forward comparison is available.

Let $Z\sim\operatorname{Bin}(n,p)$, $q=1-p$, with standard Stein operator
\begin{equation}\label{eq:bin-standard}
 \mathcal B f(k)=p(n-k)f(k+1)-qk f(k).
\end{equation}
Dividing by $q$ gives the unit-death form
\[
 \mathcal A g(k)=c(n-k)g(k+1)-kg(k),
 \qquad c:=\frac pq,
\]
so that
\[
 a=cn,
 \qquad
 \rho=-c.
\]
If $g_A$ is the unit-death solution and $f_A$ the solution of \eqref{eq:bin-standard} with the same right-hand side, then
\begin{equation}\label{eq:bin-scaling}
 f_A=\frac1q g_A.
\end{equation}
Thus the exact standard-operator envelope is
\begin{equation}\label{eq:bin-envelope}
 H_{n,p}(k)
 :=\sup_A|f_A(k)|
 =\frac{F_{k-1}Q_k}{qk\pi_k},
 \qquad 1\le k\le n.
\end{equation}

\begin{lemma}[Reflection and an interior maximizer]\label{lem:bin-boundary}
For $1\le k\le n$,
\begin{equation}\label{eq:bin-symmetry}
 H_{n,p}(k)=H_{n,q}(n-k+1).
\end{equation}
If $n\ge2$, then after replacing $p$ by $q$ if necessary one may assume $p\le q$ and choose a maximizing index in $\{1,\dots,n-1\}$.
\end{lemma}

\begin{proof}
The symmetry follows from $n-Z\sim\operatorname{Bin}(n,q)$ and
\[
 p(n-k+1)\pi_{k-1}=qk\pi_k.
\]
Assume $p\le q$.  At the endpoints,
\[
 H_{n,p}(1)=\frac{1-q^n}{np},
 \qquad
 H_{n,p}(n)=\frac{1-p^n}{nq}.
\]
If $p=q=1/2$, these two values are equal, so the claim is immediate.  Suppose henceforth that $p<q$.  The inequality $H_{n,p}(1)\ge H_{n,p}(n)$ is equivalent to
\[
 q-p\ge q^{n+1}-p^{n+1}.
\]
Since $q>p$,
\[
 \frac{q^{n+1}-p^{n+1}}{q-p}
 =\sum_{j=0}^n q^{n-j}p^j
 \le(p+q)^n=1.
\]
Thus the upper endpoint cannot be the unique maximizer; if it maximizes, then $k=1$ maximizes as well and is interior for the forward comparison when $n\ge2$.
\end{proof}

\begin{lemma}\label{lem:bin-one}
For every $n\ge1$ and $p\in(0,1)$,
\begin{equation}\label{eq:bin-one}
 C_0^{\rm Bin}(n,p):=\sup_A\norm{f_A}_\infty\le1.
\end{equation}
For $n=1$, equality holds.
\end{lemma}

\begin{proof}
The binomial balance identity gives
\[
 (j-np)\pi_j=q\{j\pi_j-(j+1)\pi_{j+1}\}.
\]
Summing over $j\ge k$ yields
\[
 \operatorname{Cov}(Z,\1{Z\ge k})=qk\pi_k.
\]
On the other hand,
\[
 \operatorname{Cov}(Z,\1{Z\ge k})
 =F_{k-1}Q_k
 \bigl(\E[Z\mid Z\ge k]-\E[Z\mid Z\le k-1]\bigr),
\]
and the conditional-mean difference is at least one.  Hence $F_{k-1}Q_k\le qk\pi_k$ and \eqref{eq:bin-one} follows from \eqref{eq:bin-envelope}.
\end{proof}

\begin{corollary}[Binomial solution factor]\label{cor:binomial}
Let $n\ge1$ and $p\in(0,1)$.  Then
\begin{equation}\label{eq:bin-coarse}
 C_0^{\rm Bin}(n,p)
 \le
 \min\left\{1,\frac1{\sqrt{2(n+1)pq}}\right\}.
\end{equation}
For $n\ge2$, put $s:=\min\{p,q\}$ and $t:=\max\{p,q\}$ and define
\begin{equation}\label{eq:bin-kappa}
 \kappa_{n,p}^{\rm Bin}
 :=\min_{1\le k\le n-1}
 \left[
 \{k-(n+1)s\}^2
 +\left\{
 \frac{st(n+1-2k)}{tk+s(n-k+1)}
 \right\}^2
 \right].
\end{equation}
Then
\begin{equation}\label{eq:bin-refined}
 C_0^{\rm Bin}(n,p)
 \le
 \min\left\{1,
 \frac1{\sqrt{2(n+1)pq+\kappa_{n,p}^{\rm Bin}}}\right\}.
\end{equation}
\end{corollary}

\begin{proof}
The case $n=1$ follows from Lemma~\ref{lem:bin-one}.  Let $n\ge2$.  By Lemma~\ref{lem:bin-boundary}, after reflection if necessary we may work with success probability $s\le1/2$ and choose an interior maximizing index.  Put $c_0=s/t$, so that the unit-death parameters are $a=c_0n$ and $\rho=-c_0$.  Then $a-\rho=c_0(n+1)$.  Theorem~\ref{thm:affine-max}(a) and the scaling \eqref{eq:bin-scaling} give
\[
 C_0^{\rm Bin}(n,p)
 \le\frac1{t\sqrt{2c_0(n+1)}}
 =\frac1{\sqrt{2(n+1)pq}},
\]
which combines with Lemma~\ref{lem:bin-one} to prove \eqref{eq:bin-coarse}.

For the refined bound, Theorem~\ref{thm:affine-max}(b), with $a=c_0n$ and $\rho=-c_0$, gives directly
\[
 H_{\rm unit}^{-2}\ge
 2c_0(n+1)
 +\{(1+c_0)k-c_0(n+1)\}^2
 +\left\{\frac{c_0(n+1-2k)}{k+c_0(n-k+1)}\right\}^2.
\]
Since $H_{n,s}=H_{\rm unit}/t$, multiplication by $t^2$ yields
\[
 H_{n,s}(k)^{-2}
 \ge2(n+1)st+\{k-(n+1)s\}^2
 +\left\{\frac{st(n+1-2k)}{tk+s(n-k+1)}\right\}^2.
\]
Taking the minimum over the admissible interior lattice points proves \eqref{eq:bin-refined}.
\end{proof}

\begin{remark}[Comparison with previous binomial factors]
R\"ollin's Lemma~5.1 in \cite{Rollin2008} gives, for arbitrary indicator sets and the standard binomial Stein equation,
\begin{equation}\label{eq:Rollin}
 \norm{f_A}_\infty\le1\wedge(npq)^{-1/2}.
\end{equation}
The bound \eqref{eq:bin-coarse} is never worse and is strictly better whenever $2(n+1)pq>1$; in the central regime the ratio of the second terms tends to $1/\sqrt2$.  Ehm's classical quantity
\[
 \frac{1-p^{n+1}-q^{n+1}}{(n+1)pq}
\]
is an increment bound for arbitrary indicators (and also a singleton-solution bound), not an arbitrary-indicator $C_0$ estimate.  Choi \cite{Choi2018} obtained general Stein-factor bounds through hitting and mixing times and records an $O(\log n)$ uniform solution bound for the binomial example.  That result substantially improves the much larger general generator bound available earlier, but it does not exploit the exact indicator envelope and is on a different scale from \eqref{eq:Rollin} and from the affine optimization here.  In the literature reviewed for the present paper, we have not located a stronger universal arbitrary-indicator solution factor than \eqref{eq:Rollin}; accordingly, \eqref{eq:bin-coarse} and \eqref{eq:bin-refined} appear to be new.
\end{remark}

\begin{remark}
Unlike the nonnegative-slope case, there is no positive parameter-free correction beyond the coarse binomial term: when $p=q=1/2$ and $n$ is odd, the two correction terms in \eqref{eq:bin-kappa} vanish simultaneously at $k=(n+1)/2$.
\end{remark}

\section{Symmetric potential distributions}\label{sec:continuous-general}

We now turn to the continuous analogue of the preceding discrete theory.  Let $Y$ have density
\begin{equation}\label{eq:cont-density}
 p(x)=Z_V^{-1}e^{-V(x)},
 \qquad
 Z_V:=\int_{\R}e^{-V(x)}\dd x<\infty,
\end{equation}
where $V$ is symmetric.  We normalize $V(0)=0$ and put
\[
 p_0:=p(0)=Z_V^{-1},
 \qquad
 g:=V'.
\]
Then $p$ is even, $g$ is odd, and
\begin{equation}\label{eq:pprime}
 p'(x)=-g(x)p(x).
\end{equation}
The Stein operator is
\begin{equation}\label{eq:cont-operator}
 \mathcal A_V f(x)=f'(x)-g(x)f(x).
\end{equation}
For an integrable test function $h$, the distinguished bounded solution of
\begin{equation}\label{eq:cont-stein}
 f_h'(x)-g(x)f_h(x)=h(x)-\E h(Y)
\end{equation}
is
\begin{equation}\label{eq:cont-solution}
 f_h(x)=\frac1{p(x)}\int_{-\infty}^x\{h(t)-\E h(Y)\}p(t)\dd t
 =-\frac1{p(x)}\int_x^\infty\{h(t)-\E h(Y)\}p(t)\dd t.
\end{equation}
For $h_z=\1{(-\infty,z]}$ we write $f_z=f_{h_z}$.  Then
\begin{equation}\label{eq:indicator-solution}
 f_z(x)=
 \begin{cases}
 \displaystyle \frac{F(x)(1-F(z))}{p(x)},&x\le z,\\[2mm]
 \displaystyle \frac{F(z)(1-F(x))}{p(x)},&x>z.
 \end{cases}
\end{equation}

We use the following hierarchy of assumptions.

\begin{description}
\item[(V1)] $V\in C^2(\R)\cap C^3(\R\setminus\{0\})$ is even, $V(0)=0$, and $Z_V<\infty$.
\item[(V2)] $g(x)>0$ for $x>0$ and $g'(x)\ge0$ for $x\in\R$.
\item[(V3)] $2(g'(x))^2-g(x)g''(x)\ge0$ for $x\in\R\setminus\{0\}$.
\end{description}
By evenness, condition (V3) need only be checked on $(0,\infty)$; no value of $g''$ at the origin is required.
\begin{remark}[Relation to earlier non-normal Stein assumptions]\label{rem:continuous-assumptions-literature}
The hierarchy (V1)--(V3) is closely related to, but not identical with, assumptions used in earlier non-normal Stein frameworks.  Chatterjee and Shao \cite{ChatterjeeShao2011} consider densities of the form
\[
 p(x)=c_1\exp\{-c_0G(x)\},\qquad G(x)=\int_0^x g_0(t)\dd t,
\]
and their condition (H1) requires $g_0$ to be nondecreasing and to have the sign of its argument.  Up to the positive scaling $g=c_0g_0$, this is the monotonicity and sign content of (V2), although our strict positivity assumption on $(0,\infty)$ is slightly stronger.  Their additional conditions (H2)--(H3) are quantitative growth conditions tailored to the exchangeable-pair error bounds and should not be identified with (V3).

Eichelsbacher and L\"owe \cite{EichelsbacherLoewe2010} develop Stein's method for a substantially broader class of regular densities, formulated through the logarithmic derivative $\psi=p'/p$ and general regularity and solution-factor assumptions.  Thus their general framework is not a direct version of (V1)--(V3).  However, the power-exponential target densities arising in their statistical-mechanics applications,
\[
 p_k(x)\propto \exp\{-a_k x^{2k}\},\qquad k\ge1,
\]
are symmetric potential distributions of the present type and satisfy (V1)--(V3).  Our assumptions isolate the particular shape properties needed for the Mills-ratio arguments below rather than aiming at the full regular-density generality of that work.

The closest correspondence is with Shao and Zhang \cite{ShaoZhang2019}.  In the non-normal part of their exchangeable-pair theory the target drift is assumed to satisfy a monotonicity and sign condition, the differential inequality
\[
 2(g')^2-gg''\ge0,
\]
and a boundary decay condition for the score-density product.  In our symmetric normalization the first two structural requirements are precisely the content of (V2)--(V3), apart from the strict positivity already noted above.  Moreover, on $(0,\infty)$ condition (V3) has the useful equivalent interpretation
\begin{equation}\label{eq:reciprocal-convex}
 \left(\frac1g\right)''
 =\frac{2(g')^2-gg''}{g^3}\ge0,
\end{equation}
that is, convexity of the reciprocal drift.  Shao and Zhang additionally assume $g(x)p(x)\to0$ at the boundary.  In the present symmetric whole-line setting this is automatic under (V2)--(V3).  Indeed, with $q=1/g$ on $(0,\infty)$, one has $q>0$, $q'\le0$ and $q''\ge0$, hence $q'\uparrow0$.  Therefore $g'/g^2=-q'\to0$ and
\[
 (V-\log g)'=g\left(1-\frac{g'}{g^2}\right)\ge\frac{g}{2}
\]
eventually.  Since $g$ is positive and nondecreasing on every positive tail, $V-\log g\to\infty$, that is, $g(x)e^{-V(x)}\to0$ as $x\to\infty$.  Thus $g(x)p(x)\to0$ at both ends by symmetry, and we do not list the boundary assumption separately.
\end{remark}

\subsection{Mills-ratio calculus}

Define
\begin{equation}\label{eq:cont-mills}
 m(x):=\frac{F(x)}{p(x)},
 \qquad
 r(x):=m(-x)=\frac{1-F(x)}{p(x)},\qquad x>0.
\end{equation}
Differentiation gives
\begin{equation}\label{eq:mills-diff}
 m'(x)=1+g(x)m(x),
 \qquad
 r'(x)=g(x)r(x)-1,
\end{equation}
and
\begin{equation}\label{eq:mills-second}
 m''(x)=\{g'(x)+g(x)^2\}m(x)+g(x).
\end{equation}

\begin{proposition}[Mills bounds and their hierarchy]\label{prop:cont-mills}
Assume (V1)--(V2).  Then, for $x>0$,
\begin{equation}\label{eq:upper-mills}
 0<r(x)\le\frac1{g(x)},
\end{equation}
and consequently $m'(x)\ge0$ on $\R$.

If, in addition, (V3) holds, then
\begin{equation}\label{eq:two-sided-mills}
 \frac{g(x)}{g(x)^2+g'(x)}\le r(x)\le\frac1{g(x)},
\end{equation}
$m''(x)\ge0$ on $\R$, and
\begin{equation}\label{eq:tail-asymptotic}
 g(x)r(x)\longrightarrow1\qquad(x\to\infty).
\end{equation}
\end{proposition}

\begin{proof}
Under (V2), for $x>0$,
\[
 1-F(x)=\int_x^\infty \frac{g(t)p(t)}{g(t)}\dd t
 \le \frac1{g(x)}\int_x^\infty g(t)p(t)\dd t
 =\frac{p(x)}{g(x)},
\]
which proves \eqref{eq:upper-mills}.  Hence $m'(-x)=1-g(x)r(x)\ge0$ for $x>0$, while $m'(x)=1+g(x)m(x)>0$ for $x\ge0$.

Assume (V3) and set
\[
 \ell:=\frac{g}{g^2+g'}.
\]
A direct calculation gives
\[
 1-g\ell+\ell'
 =\frac{2(g')^2-gg''}{(g^2+g')^2}\ge0,
\]
and therefore, using $r'-gr=-1$,
\[
 \bigl(e^{-V}(r-\ell)\bigr)'
 =-e^{-V}(1-g\ell+\ell')\le0.
\]
Moreover,
\[
 e^{-V}r=Z_V(1-F)\longrightarrow0,
 \qquad
 0\le e^{-V}\ell\le\frac{e^{-V}}g\longrightarrow0;
\]
the last limit follows because $g$ is positive and nondecreasing on every positive tail.  Thus $r\ge\ell$, proving the lower bound in \eqref{eq:two-sided-mills}.

For $x\ge0$, \eqref{eq:mills-second} gives $m''(x)\ge0$.  For $x<0$, symmetry and the lower Mills bound give
\[
 m''(x)=r''(-x)
 =\{g'(-x)+g(-x)^2\}r(-x)-g(-x)\ge0.
\]
Finally,
\[
 \frac1{1+g'/g^2}\le gr\le1.
\]
Since $q=1/g$ is positive, decreasing and convex under (V3), the derivative $q'$ is increasing and has a limit $L\le0$.  If $L<0$, then $q$ would eventually decrease at a fixed negative rate and become negative, contradicting $q>0$.  Thus $L=0$, so $q'\uparrow0$ and hence $g'/g^2=-q'\to0$.  The squeeze theorem proves \eqref{eq:tail-asymptotic}.
\end{proof}

\begin{remark}[Position of the Mills bounds in the literature]\label{rem:mills-literature}
The upper estimate $r(x)\le 1/g(x)$ is the familiar monotone-score Mills bound.  The refined lower estimate in \eqref{eq:two-sided-mills} is the form naturally produced by the differential condition $2(g')^2-gg''\ge0$, which is the continuous structural condition already present in the general framework of Shao and Zhang \cite{ShaoZhang2019}.  We do not claim the individual tail inequalities as a separate novelty.  Their role here is to make the hierarchy of assumptions transparent and to isolate the two consequences needed below: $m''\ge0$ and the sharp tail relation $gr\to1$.
\end{remark}

We shall also use the symmetric Mills identity
\begin{equation}\label{eq:cont-mills-identity}
 F(-x)m'(x)+F(x)m'(-x)=1,
\end{equation}
which follows immediately from \eqref{eq:mills-diff}, oddness of $g$, and symmetry of $F$.

\subsection{Indicator factors: a hierarchy of conclusions}

Because $m$ is nondecreasing under (V1)--(V2), the two branches in \eqref{eq:indicator-solution} meet at their common maximum.  Hence
\begin{equation}\label{eq:cont-envelope}
 \norm{f_z}_\infty
 =\frac{F(z)(1-F(z))}{p(z)}.
\end{equation}
Put
\begin{equation}\label{eq:CK}
 C_{\mathrm K}(V)
 :=\sup_{z\in\R}\norm{f_z}_\infty
 =\sup_{z\in\R}\frac{F(z)(1-F(z))}{p(z)}.
\end{equation}
For indicator solutions, $f_z'$ is understood on $\R\setminus\{z\}$ and its norm is the essential supremum, equivalently the supremum over the two smooth branches.  The next theorem makes the hierarchy explicit.  It separates convex unimodality, reciprocal-drift convexity, and the additional sharpness information.

\begin{theorem}[Continuous indicator Stein factors]\label{thm:continuous-main}
Assume (V1)--(V2).

\emph{(a) Basic bounds.} For every $z\in\R$,
\begin{equation}\label{eq:basic-first-order}
 \norm{f_z'}_\infty\le1,
 \qquad
 \norm{g f_z}_\infty\le \max\{F(z),1-F(z)\}\le1.
\end{equation}
Moreover, for every $a>0$,
\begin{equation}\label{eq:CK-a-bound}
 \frac1{4p_0}\le C_{\mathrm K}(V)
 \le
 \max\left\{
 \frac{e^{V(a)}}{4p_0},\frac1{g(a)}
 \right\}.
\end{equation}
Equivalently,
\begin{equation}\label{eq:CK-BV}
 C_{\mathrm K}(V)\le B_V
 :=\inf_{a>0}\max\left\{\frac{e^{V(a)}}{4p_0},\frac1{g(a)}\right\}<\infty.
\end{equation}

\emph{(b) Sharp first-order factors.} If (V3) also holds, then $x\mapsto g(x)f_z(x)$ is nondecreasing and, for every $z\in\R$,
\begin{align}
 \norm{g f_z}_\infty
 &=\max\{F(z),1-F(z)\},\label{eq:fixed-z-drift}\\
 \norm{f_z'}_\infty
 &=\max\bigl\{(1-F(z))m'(z),\,F(z)m'(-z)\bigr\}.
 \label{eq:fixed-z-derivative}
\end{align}
Consequently,
\begin{equation}\label{eq:sharp-first-order}
 \sup_{z\in\R}\norm{f_z'}_\infty=1,
 \qquad
 \sup_{z\in\R}\norm{g f_z}_\infty=1.
\end{equation}

\emph{(c) Optimal solution factor under a shape condition.} Define
\begin{equation}\label{eq:K-def}
 K(x):=4p_0\{2F(x)-1\}-g(x),\qquad x>0.
\end{equation}
For part~(c), assume, in addition to (V1)--(V2), the one-crossing condition
\begin{equation*}
 \tag{OC}\label{eq:OC}
 K>0\text{ on }(0,x_*),\qquad K(x_*)=0,\qquad K<0\text{ on }(x_*,\infty)
\end{equation*}
for some $x_*>0$.  Then
\begin{equation}\label{eq:optimal-cont-factor}
 C_{\mathrm K}(V)=\frac1{4p(0)}.
\end{equation}
\end{theorem}

\begin{proof}
The lower bound in \eqref{eq:CK-a-bound} is the value of \eqref{eq:cont-envelope} at $z=0$.  By symmetry take $z\ge0$.  For $0\le z\le a$,
\[
 \frac{F(z)(1-F(z))}{p(z)}\le\frac{e^{V(a)}}{4p_0},
\]
whereas for $z\ge a$, Proposition~\ref{prop:cont-mills} gives
\[
 \frac{F(z)(1-F(z))}{p(z)}=F(z)r(z)
 \le r(z)\le\frac1{g(z)}\le\frac1{g(a)}.
\]
This proves \eqref{eq:CK-a-bound}--\eqref{eq:CK-BV}.

For $x\ne z$, differentiation of \eqref{eq:indicator-solution} gives
\[
 f_z'(x)=
 \begin{cases}
 (1-F(z))m'(x),&x<z,\\
 -F(z)m'(-x),&x>z.
 \end{cases}
\]
By Proposition~\ref{prop:cont-mills}, $m'\ge0$, and the two terms in \eqref{eq:cont-mills-identity} are nonnegative.  If $x<z$, then $1-F(z)=F(-z)\le F(-x)$, so
\[
 0\le f_z'(x)\le F(-x)m'(x)\le1.
\]
If $x>z$, then $F(z)\le F(x)$ and therefore $-1\le f_z'(x)\le0$.  Thus $\|f_z'\|_\infty\le1$.  The Stein equation gives, on either side of $z$,
\[
 -(1-F(z))\le g(x)f_z(x)\le F(z),
\]
which proves the sharper drift bound in \eqref{eq:basic-first-order}.

Assume (V3).  Since $m''\ge0$,
\[
 \frac{\dd}{\dd x}\{g(x)f_z(x)\}
 =
 \begin{cases}
 (1-F(z))m''(x),&x<z,\\
 F(z)m''(-x),&x>z,
 \end{cases}
\]
on the two open branches.  Since $f_z$ and $g$ are continuous, $g f_z$ is continuous at $z$; hence $g f_z$ is nondecreasing on all of $\R$.  The relation $gr\to1$ gives the endpoint limits $-(1-F(z))$ and $F(z)$, proving \eqref{eq:fixed-z-drift}.

The same monotonicity of $m'$ shows that the two one-sided extrema of $|f_z'|$ occur at $z$, which gives \eqref{eq:fixed-z-derivative}.  Its two entries add to $1$ by \eqref{eq:cont-mills-identity}.  As $z\to\infty$, $m'(-z)=1-g(z)r(z)\to0$, so the first entry tends to $1$.  This proves the first equality in \eqref{eq:sharp-first-order}; the second follows immediately from \eqref{eq:fixed-z-drift}.

Finally set
\[
 J(x):=p(x)-4p_0F(x)(1-F(x)),\qquad x\ge0.
\]
Then $J'(x)=p(x)K(x)$, while $J(0)=0$ and $J(x)\to0$ as $x\to\infty$.  Under \eqref{eq:OC}, $J$ first increases and then decreases to $0$, hence $J\ge0$.  Equivalently,
\[
 \frac{F(x)(1-F(x))}{p(x)}\le\frac1{4p_0},
\]
with equality at $x=0$, which proves part~(c).
\end{proof}

\begin{remark}[Sharpness and comparison with earlier continuous Stein factors]\label{rem:continuous-comparison}
Theorem~\ref{thm:continuous-main} should be compared with earlier non-normal Stein results at two distinct levels: the size of the available upper bounds and the question whether the corresponding Stein factors are actually sharp.

For the power-exponential targets arising in mean-field statistical mechanics, Eichelsbacher and L\"owe \cite{EichelsbacherLoewe2010} proved the indicator bounds
\[
 \norm{f_z}_\infty\le\frac1{2p(0)},
 \qquad \norm{f_z'}_\infty\le1,
 \qquad \norm{g f_z}_\infty\le1.
\]
They explicitly observed that the solution constant $1/(2p(0))$ is not optimal and remarked that a different argument would lead to optimal constants, but this optimization was not carried out there.  Shao and Zhang \cite{ShaoZhang2019} subsequently treated a substantially broader class of non-normal targets and, in the present symmetric normalization, obtained
\[
 0\le f_z(x)\le\frac1{p(0)},
 \qquad \norm{f_z'}_\infty\le1,
 \qquad \norm{g f_z}_\infty\le1,
\]
together with monotonicity of $g f_z$.  These estimates are used as Stein-factor upper bounds in their Berry--Esseen analysis; the solution factor is not optimized there, and the constants $1$ are not identified there as sharp uniform suprema.

For the independent kernel approach of Rapin and Swan \cite{RapinSwan2026} and the detailed comparison on the Subbotin family, see Remark~\ref{rem:power-exponential}.  The point needed here is only that, for the structural class considered in this section, the first-order sharpness statements follow directly from the Mills identities above and the optimal solution factor follows from the one-crossing test in part~(c).

Thus parts~(b)--(c) identify exact Stein factors rather than merely provide upper bounds: under (V3) both first-order constants equal $1$, and under (OC) the solution factor is exactly $1/(4p(0))$.  For the power-exponential family this replaces the earlier bound $1/(2p(0))$ by its exact optimal value $1/(4p(0))$.
\end{remark}

\begin{remark}[The explicit basic bound]
The two terms in \eqref{eq:CK-a-bound} move in opposite directions.  Since
\[
 \frac{\dd}{\dd a}\{e^{V(a)}g(a)\}=e^{V(a)}\{g(a)^2+g'(a)\}>0,
\]
and $e^{V(a)}g(a)$ increases from $0$ to $\infty$ as $a$ runs from $0$ to $\infty$, there is a unique $a_V>0$ satisfying
\begin{equation}\label{eq:aV}
 e^{V(a_V)}g(a_V)=4p_0.
\end{equation}
Thus
\[
 B_V=\frac1{g(a_V)}=\frac{e^{V(a_V)}}{4p_0}.
\]
This provides a concrete finite solution factor without any one-crossing assumption.
\end{remark}

\begin{remark}[What one crossing does---and does not mean]\label{rem:OC-not-necessary}
Condition \eqref{eq:OC} is not necessary for the optimal value in \eqref{eq:optimal-cont-factor}.  The exact necessary and sufficient condition is simply
\begin{equation}\label{eq:exact-sharp-condition}
 J(x)=p(x)-4p_0F(x)(1-F(x))\ge0,
 \qquad x\ge0,
\end{equation}
or, equivalently,
\[
 \int_0^x p(t)K(t)\dd t\ge0
 \qquad(x\ge0).
\]
This condition is global and is essentially a restatement of the desired optimal bound.  The role of (OC) is to provide a transparent local shape criterion implying it.
\end{remark}

\section{Why the continuous assumptions are separated}\label{sec:assumptions}

The normalization $V(0)=0$ in (V1) is harmless: adding a constant to $V$ does not change the normalized density.  The substantive assumptions are symmetry, smoothness, integrability, convex unimodality, reciprocal-drift convexity, and---only for the closed-form optimal solution factor---additional global shape information.  The next two examples show that the last two requirements are genuinely additional.

\begin{example}[Strict convexity of $V$ does not imply (V3)]\label{ex:V3-fail}
Let
\[
 V(x)=\frac{x^2}{2}-\frac{x^4}{4}+\frac{x^6}{6},
 \qquad
 g(x)=x-x^3+x^5.
\]
For $x>0$, $g(x)>0$, and
\[
 g'(x)=1-3x^2+5x^4>0
\]
for all $x$, so $V$ is strictly convex and (V2) holds.  However
\[
 2(g')^2-gg''
 =2\bigl(15x^8-17x^6+6x^4-3x^2+1\bigr),
\]
which equals $-3/8$ at $x=1/\sqrt2$.  Thus (V3) is genuinely stronger than ordinary convexity of $V$.
\end{example}

\begin{example}[(V1)--(V3) do not force the optimal origin value]\label{ex:OC-fail}
Consider the explicit potential
\[
 V(x)=\log\cosh x.
\]
Then
\[
 p(x)=\frac1{\pi\cosh x},
 \qquad
 p(0)=\frac1\pi,
 \qquad
 g(x)=\tanh x.
\]
Moreover,
\[
 g'(x)=\operatorname{sech}^2x,
 \qquad
 g''(x)=-2\operatorname{sech}^2x\tanh x,
\]
and hence
\[
 2(g')^2-gg''=2\operatorname{sech}^2x>0.
\]
Thus (V1)--(V3) hold strictly.  The distribution function is
\[
 F(x)=\frac12+\frac1\pi\arctan(\sinh x).
\]
By Proposition~\ref{prop:cont-mills}, or directly from the explicit formula,
\[
 \frac{1-F(x)}{p(x)}\longrightarrow1,
\]
so
\[
 \frac{F(x)(1-F(x))}{p(x)}\longrightarrow1.
\]
At the origin, however,
\[
 \frac1{4p(0)}=\frac\pi4<1.
\]
Therefore the core conditions (V1)--(V3) do not imply that the optimal solution factor is attained at the origin.  Some additional global information, such as \eqref{eq:exact-sharp-condition} or the sufficient one-crossing criterion, is genuinely needed.
\end{example}

\section{Even-power and Subbotin targets}\label{sec:continuous-examples}

We finally return to the family that originally motivated the continuous part.  In the statistical-mechanics applications of Eichelsbacher and L\"owe \cite{EichelsbacherLoewe2010}, the limiting densities considered there are proportional to
\[
 \exp\left\{-\frac{\mu |x|^{2k}}{(2k)!}\right\},
 \qquad k\ge1.
\]
With $\lambda=\mu/(2k)!$, write
\begin{equation}\label{eq:power-exp}
 p_{k,\lambda}(x)
 =\frac{k\lambda^{1/(2k)}}{\Gamma(1/(2k))}e^{-\lambda |x|^{2k}},
 \qquad k\ge1,\quad \lambda>0.
\end{equation}
For $x>0$ the corresponding drift is $g(x)=2k\lambda x^{2k-1}$, and
\[
 2(g')^2-gg''=(2k)^3(2k-1)\lambda^2x^{4k-4}\ge0.
\]
Thus (V1)--(V3) hold.  Moreover, with $K$ as in \eqref{eq:K-def},
\[
 K''(x)=-8p(0)g(x)p(x)-g''(x)<0,
 \qquad x>0,
\]
and $K'(0+)>0$: for $k\ge2$ this is immediate, while for $k=1$ it equals
$2\lambda(4/\pi-1)$.  Hence (OC) holds for every $k\ge1$, and Theorem~\ref{thm:continuous-main} gives
\begin{equation}\label{eq:power-exp-factors}
 \sup_{z\in\R}\norm{f_z}_\infty
 =\frac{\Gamma(1/(2k))}{4k\lambda^{1/(2k)}},
 \qquad
 \sup_{z\in\R}\norm{f_z'}_\infty=1,
 \qquad
 \sup_{z\in\R}\norm{2k\lambda\,\operatorname{sgn}(x)|x|^{2k-1}f_z(x)}_\infty=1.
\end{equation}
All three constants are sharp.

\begin{corollary}[Quartic Stein factors]\label{cor:quartic}
For
\[
 p_\lambda(x)=\frac{2\lambda^{1/4}}{\Gamma(1/4)}e^{-\lambda x^4},
 \qquad \lambda>0,
\]
one has
\[
 \sup_{z\in\R}\norm{f_z}_\infty
 =\frac{\Gamma(1/4)}{8\lambda^{1/4}},
 \qquad
 \sup_{z\in\R}\norm{f_z'}_\infty=1,
 \qquad
 \sup_{z\in\R}\norm{4\lambda x^3f_z(x)}_\infty=1.
\]
\end{corollary}

\begin{proof}
This is the case $k=2$ of \eqref{eq:power-exp}--\eqref{eq:power-exp-factors}.
\end{proof}

\begin{remark}[Comparison for the even-power targets]\label{rem:even-power-comparison}
The family \eqref{eq:power-exp}, rather than the full Subbotin family considered below, was the original motivation for the continuous examples in this paper.  Eichelsbacher and L\"owe \cite{EichelsbacherLoewe2010} obtained for these targets the indicator estimate $\norm{f_z}_\infty\le1/(2p(0))$ together with the first-order bounds by $1$, and explicitly noted that the solution constant was not optimal.  The general non-normal frameworks of Chatterjee and Shao \cite{ChatterjeeShao2011} and Shao and Zhang \cite{ShaoZhang2019} likewise provide Stein factors for approximation purposes.  Formula \eqref{eq:power-exp-factors} identifies instead the exact zeroth-order constant $1/(4p(0))$ and the sharpness of the two first-order constants on this family.

During the final preparation of the manuscript, Rapin and Swan \cite{RapinSwan2026} posted their recent work.  Their substantially broader kernel calculus includes the Subbotin family and, on the even-power overlap, yields the same sharp zeroth- and first-derivative constants.  Their Subbotin example prompted us to examine whether the elementary Mills-ratio and one-crossing calculation above extends beyond even integer powers.  We return to this point in Remark~\ref{rem:power-exponential}.
\end{remark}

\begin{remark}[The critical Curie--Weiss specialization]\label{rem:CW}
The quartic critical Curie--Weiss law is already contained in Corollary~\ref{cor:quartic}: choosing $\lambda=1/12$ gives
\[
 p_*(x)=\frac{\sqrt2}{3^{1/4}\Gamma(1/4)}e^{-x^4/12}
\]
and hence
\[
 \sup_z\norm{f_z}_\infty
 =\frac{12^{1/4}\Gamma(1/4)}8,
 \qquad
 \sup_z\norm{f_z'}_\infty=1,
 \qquad
 \sup_z\norm{\tfrac13x^3f_z(x)}_\infty=1.
\]
Thus no separate argument is needed for the classical critical Curie--Weiss normalization.
\end{remark}

\begin{remark}[From even powers to the Subbotin family]\label{rem:power-exponential}
Before carrying out the extension, it is useful to record the classical parametrization.  The exponential-power family goes back to Subbotin's 1923 paper \emph{On the Law of Frequency of Error} \cite{Subbotin1923} and, in the normalization used by Rapin and Swan \cite{RapinSwan2026}, is written
\[
 p_\beta(x)=C_\beta
 \exp\left\{-\frac{|x|^\beta}{\beta(\beta-1)}\right\},
 \qquad \beta>1.
\]
Motivated by their treatment, we consider the slightly more flexible scaled family
\begin{equation}\label{eq:subbotin-scaled}
 p_{\beta,\lambda}(x)
 =\frac{\beta\lambda^{1/\beta}}{2\Gamma(1/\beta)}
   e^{-\lambda |x|^\beta},
 \qquad \beta>1,\quad \lambda>0.
\end{equation}
The classical normalization corresponds to $\lambda=1/\{\beta(\beta-1)\}$, while $\beta=2k$ recovers the even-power family \eqref{eq:power-exp}.

We first take $\beta\ge2$.  Then $V(x)=\lambda|x|^\beta$ satisfies (V1)--(V3).  For $x>0$,
\[
 g(x)=\beta\lambda x^{\beta-1},
 \qquad
 g'(x)=\beta(\beta-1)\lambda x^{\beta-2},
 \qquad
 2(g')^2-gg''=\beta^3(\beta-1)\lambda^2x^{2\beta-4}>0.
\]
Moreover,
\[
 K''(x)=-8p(0)g(x)p(x)-g''(x)<0.
\]
If $\beta>2$, then $K'(0+)=8p(0)^2>0$, whereas for $\beta=2$,
\[
 K'(0+)=8p(0)^2-2\lambda
       =2\lambda\left(\frac4\pi-1\right)>0.
\]
Thus (OC) holds throughout the range $\beta\ge2$, and Theorem~\ref{thm:continuous-main} yields
\begin{equation}\label{eq:subbotin-scaled-factors}
 \sup_z\norm{f_z}_\infty
 =\frac{\Gamma(1/\beta)}{2\beta\lambda^{1/\beta}},
 \qquad
 \sup_z\norm{f_z'}_\infty=1,
 \qquad
 \sup_z\norm{\beta\lambda\,\operatorname{sgn}(x)|x|^{\beta-1}f_z(x)}_\infty=1.
\end{equation}

In the classical normalization, the first constant in \eqref{eq:subbotin-scaled-factors} becomes
\[
 \frac{(\beta(\beta-1))^{1/\beta}\Gamma(1/\beta)}{2\beta},
\]
which agrees with the sharp Kolmogorov solution factor in equation~(3.24) of Rapin and Swan \cite{RapinSwan2026}; the sharp first-derivative factor $1$ agrees with their second identity there.  Moreover,
\[
 \rho_{p_\beta}(x)
 =-\frac{\operatorname{sgn}(x)|x|^{\beta-1}}{\beta-1}
 =-g(x),
\]
so \eqref{eq:subbotin-scaled-factors} also identifies the corresponding score-weighted factor as sharp.  Rapin and Swan obtain the upper bound by $1$ for this quantity in their equation~(3.25).  Their kernel theory is substantially broader in derivative order and test-function regularity; the point of the calculation here is that the common Kolmogorov quantities on $\beta\ge2$ follow directly from the elementary potential conditions and the one-crossing criterion.

The interval $1<\beta<2$ is different.  In that range $V(x)=\lambda|x|^\beta$ is not $C^2$ at the origin, so (V1)--(V3) are not all available as stated and Theorem~\ref{thm:continuous-main} does not apply directly.  In particular, the preceding origin-maximization argument does not determine the exact indicator envelope.  The comparison above with Rapin and Swan leaves this exact maximization problem open in the form needed here.  We therefore treat the range separately in Proposition~\ref{prop:subbotin-below-two}.
\end{remark}

The transformed distribution $P_a(u)^2$ appearing below belongs to the exponentiated-gamma setting; related exponentiated generalized gamma models and hazard-rate comparisons for gamma systems have been studied in \cite{CordeiroOrtegaSilva2011,BalakrishnanZhao2013}.  Although the exponentiated gamma distribution and several qualitative forms of its hazard rate have been studied previously, we have not found in the literature the following precise result for the boundary regime in which the hazard converges to a finite positive level.

\begin{proposition}[Subbotin targets below the Gaussian threshold]\label{prop:subbotin-below-two}
Let $1<\beta<2$ and $\lambda>0$, and let $Y$ have density $p_{\beta,\lambda}$ from \eqref{eq:subbotin-scaled}.  Put
\[
 a:=\frac1\beta\in\left(\frac12,1\right),
 \qquad
 P_a(u):=\frac{\gamma(a,u)}{\Gamma(a)},
 \qquad u\ge0,
\]
where $\gamma(a,u)=\int_0^u t^{a-1}e^{-t}\dd t$ is the lower incomplete gamma function.  Then the exact indicator-solution envelope
\[
 H_{\beta,\lambda}(x):=\sup_{z\in\R}|f_z(x)|
 =\frac{F(x)\{1-F(x)\}}{p_{\beta,\lambda}(x)}
\]
has exactly two global maximizers $\pm x_{\beta,\lambda}$, where
\begin{equation}\label{eq:subbotin-offcentre-location}
 x_{\beta,\lambda}
 =\left(\frac{u_\beta}{\lambda}\right)^{1/\beta}
\end{equation}
and $u_\beta$ is the unique positive solution of
\begin{equation}\label{eq:subbotin-critical-equation}
 2P_a(u)P_a'(u)=1-P_a(u)^2.
\end{equation}
Consequently,
\begin{align}
 C_{\mathrm K}(\beta,\lambda)
 :=\sup_{z\in\R}\norm{f_z}_\infty
 &=\frac{\Gamma(a)}{2\beta\lambda^a}
   e^{u_\beta}\{1-P_a(u_\beta)^2\}\label{eq:subbotin-below-factor}\\
 &=\frac{P_a(u_\beta)u_\beta^{a-1}}{\beta\lambda^a}
 >\frac1{4p_{\beta,\lambda}(0)}.\label{eq:subbotin-below-factor-alt}
\end{align}
Moreover, with
\[
 g(x)=\beta\lambda\,\operatorname{sgn}(x)|x|^{\beta-1},
\]
the two first-order factors remain sharp:
\begin{equation}\label{eq:subbotin-below-first-order}
 \sup_{z\in\R}\norm{f_z'}_\infty=1,
 \qquad
 \sup_{z\in\R}\norm{g f_z}_\infty=1.
\end{equation}
\end{proposition}

\begin{proof}
For fixed $x$, formula~\eqref{eq:indicator-solution} shows that the branch with $z\ge x$ decreases in $z$, whereas the branch with $z<x$ increases in $z$.  Hence the supremum over $z$ is attained at $z=x$ (with the usual one-sided interpretation) and equals $F(x)\{1-F(x)\}/p_{\beta,\lambda}(x)$.  By symmetry it is therefore enough to work on $x\ge0$.  With $u=\lambda x^\beta$ one has
\[
 F(x)=\frac{1+P_a(u)}2,
 \qquad
 p_{\beta,\lambda}(x)
 =\frac{\beta\lambda^a}{2\Gamma(a)}e^{-u}.
\]
Hence
\begin{equation}\label{eq:subbotin-Phi}
 H_{\beta,\lambda}(x)
 =\frac{\Gamma(a)}{2\beta\lambda^a}\,\Phi_a(u),
 \qquad
 \Phi_a(u):=e^u\{1-P_a(u)^2\}.
\end{equation}
We prove that $\Phi_a$ has a unique maximizer on $(0,\infty)$.

Set
\[
 D_a(u):=\frac{2P_a(u)P_a'(u)}{1-P_a(u)^2}.
\]
Then
\begin{equation}\label{eq:Phi-log-derivative}
 \frac{\Phi_a'(u)}{\Phi_a(u)}=1-D_a(u).
\end{equation}
Since $P_a(u)\sim u^a/\Gamma(a+1)$ and $P_a'(u)=u^{a-1}e^{-u}/\Gamma(a)$ as $u\downarrow0$,
\begin{equation}\label{eq:D-zero}
 D_a(u)\longrightarrow0\qquad(u\downarrow0),
\end{equation}
because $2a-1>0$.

Introduce also
\[
 R_a(u):=\frac{uP_a'(u)}{P_a(u)}
 =\frac{u^ae^{-u}}{\gamma(a,u)}.
\]
The recurrence
\[
 \gamma(a+1,u)=a\gamma(a,u)-u^ae^{-u}
\]
gives
\[
 R_a(u)=a-\frac{\gamma(a+1,u)}{\gamma(a,u)}.
\]
The quotient on the right is the mean of $t$ with respect to the probability measure proportional to $t^{a-1}e^{-t}\1{(0,u)}(t)\dd t$; hence
\begin{equation}\label{eq:R-bounds}
 a-u<R_a(u)<a.
\end{equation}
Moreover, direct differentiation yields
\begin{equation}\label{eq:R-derivative}
 uR_a'(u)=R_a(u)\{a-u-R_a(u)\}<0.
\end{equation}
Moreover, $R_a(u)\to a$ as $u\downarrow0$ and $R_a(u)\to0$ as $u\to\infty$.  Thus $R_a$ decreases strictly from $a$ to $0$.  Since $1-a\in(0,a)$, there is a unique $u_0=u_0(a)>0$ such that
\begin{equation}\label{eq:R-u0}
 R_a(u_0)=1-a.
\end{equation}

A logarithmic differentiation of $D_a$ gives the key identity
\begin{equation}\label{eq:D-log-derivative}
 \frac{D_a'(u)}{D_a(u)}
 =D_a(u)-1+\frac{R_a(u)-(1-a)}u.
\end{equation}
Therefore, at every point where $D_a(u)=1$,
\begin{equation}\label{eq:D-crossing-slope}
 D_a'(u)=\frac{R_a(u)-(1-a)}u.
\end{equation}
Thus a crossing of the level $1$ to the left of $u_0$ can only be upward, whereas a crossing to the right of $u_0$ can only be downward.  If $D_a(u_0)=1$, then \eqref{eq:D-log-derivative} and \eqref{eq:R-u0} give $D_a'(u_0)=0$, and differentiating once more yields
\[
 D_a''(u_0)=\frac{R_a'(u_0)}{u_0}<0.
\]
Hence such a contact would be a strict local maximum of $D_a$ at level $1$.

It remains to see that a genuine crossing must occur.  Since $0<a<1$, for $t\ge u$ one has $t^{a-1}\le u^{a-1}$, and therefore the upper incomplete gamma function satisfies
\[
 \Gamma(a,u)=\int_u^\infty t^{a-1}e^{-t}\dd t
 \le u^{a-1}e^{-u}.
\]
Consequently,
\[
 0<\Phi_a(u)
 =e^u\{1-P_a(u)\}\{1+P_a(u)\}
 \le\frac{2u^{a-1}}{\Gamma(a)}\longrightarrow0.
\]
On the other hand, \eqref{eq:D-zero} and \eqref{eq:Phi-log-derivative} show that $\Phi_a$ is increasing for small positive $u$.  Since $\Phi_a(u)\to0$, its derivative must be negative somewhere, and hence $D_a>1$ somewhere.  Define
\[
 u_\beta:=\inf\{u>0:D_a(u)\ge1\}.
\]
By continuity, $D_a(u_\beta)=1$.  This first contact cannot occur to the right of $u_0$, because \eqref{eq:D-crossing-slope} would then give $D_a'(u_\beta)<0$, whereas $D_a<1$ immediately to its left.  It cannot occur at $u_0$ either: in that case the tangency calculation above would make $u_0$ a strict local maximum of $D_a$ at level $1$, so $D_a<1$ just to the right of $u_0$.  If there were no later passage above $1$, then $D_a\le1$ from that point onward and \eqref{eq:Phi-log-derivative} would force $\Phi_a$ to be nondecreasing eventually, contradicting $\Phi_a(u)\to0$.  Any later first passage above $1$ would occur to the right of $u_0$ with nonnegative slope, contradicting \eqref{eq:D-crossing-slope}.  Thus $u_\beta<u_0$ and $D_a'(u_\beta)>0$, so the crossing is upward.

There is no further root of $D_a-1$ to the left of $u_0$: after the first upward crossing, any return from values above $1$ to values below $1$ would require a root with nonpositive derivative, whereas every root to the left of $u_0$ has positive derivative by \eqref{eq:D-crossing-slope}.  If there were a first further root $v>u_0$, then \eqref{eq:D-crossing-slope} would give $D_a'(v)<0$, so $D_a<1$ immediately after $v$.  No later root could return $D_a$ above $1$, because every root to the right of $u_0$ again has negative derivative.  Hence $D_a<1$ for all sufficiently large $u$, and \eqref{eq:Phi-log-derivative} would make the positive function $\Phi_a$ eventually increasing, contradicting $\Phi_a(u)\to0$.  Therefore $u_\beta$ is the unique solution of \eqref{eq:subbotin-critical-equation}, and $\Phi_a$ increases on $(0,u_\beta)$ and decreases on $(u_\beta,\infty)$.  Equations \eqref{eq:subbotin-offcentre-location} and \eqref{eq:subbotin-below-factor} follow.  At the maximizing point, \eqref{eq:subbotin-critical-equation} gives
\[
 e^{u_\beta}\{1-P_a(u_\beta)^2\}
 =2P_a(u_\beta)\frac{u_\beta^{a-1}}{\Gamma(a)},
\]
which proves the second equality in \eqref{eq:subbotin-below-factor-alt}.  Since $\Phi_a(0)=1$ and $\Phi_a$ initially increases, its maximum is strictly larger than $1$; using $1/(4p_{\beta,\lambda}(0))=\Gamma(a)/(2\beta\lambda^a)$ proves the strict inequality.

It remains to verify the first-order factors.  On $(0,\infty)$ the drift $g(x)=\beta\lambda x^{\beta-1}$ is positive and increasing, so the usual upper Mills argument gives $r(x)\le1/g(x)$.  Furthermore, with
\[
 \ell(x):=\frac{g(x)}{g(x)^2+g'(x)},
\]
one has
\[
 1-g\ell+\ell'
 =\frac{2(g')^2-gg''}{(g^2+g')^2}>0,
 \qquad x>0,
\]
because
\[
 2(g')^2-gg''
 =\beta^3(\beta-1)\lambda^2x^{2\beta-4}>0.
\]
As in Proposition~\ref{prop:cont-mills},
\[
 \bigl(e^{-V}(r-\ell)\bigr)'=-e^{-V}(1-g\ell+\ell')<0
 \qquad (x>0),
\]
and both $e^{-V}r$ and $e^{-V}\ell$ tend to $0$ at infinity; hence $r\ge\ell$ on $(0,\infty)$.  It follows that $m''(x)>0$ for $x>0$, while for $x<0$ symmetry gives
\[
 m''(x)=\{g'(-x)+g(-x)^2\}r(-x)-g(-x)\ge0.
\]
Although $m''$ need not exist at the origin, $g$ is continuous there with $g(0)=0$, and $m'(x)=1+g(x)m(x)$ is continuous with $m'(0)=1$.  Since $m'$ is nondecreasing on each open half-line, it is therefore nondecreasing on all of $\R$.  The proof of Theorem~\ref{thm:continuous-main}(b) now applies verbatim on the two open half-lines and extends across the origin by continuity.  In particular, $gf_z$ is nondecreasing and
\[
 \norm{g f_z}_\infty=\max\{F(z),1-F(z)\},
\]
while
\[
 \norm{f_z'}_\infty
 =\max\bigl\{(1-F(z))m'(z),\,F(z)m'(-z)\bigr\}.
\]
Finally, the two Mills bounds give
\[
 \frac{1}{1+g'(x)/g(x)^2}\le g(x)r(x)\le1,
 \qquad
 \frac{g'(x)}{g(x)^2}=\frac{\beta-1}{\beta\lambda x^\beta}\longrightarrow0,
\]
so $g(x)r(x)\to1$.  The symmetric Mills identity shows that the two terms in the last maximum add to $1$.  Letting $|z|\to\infty$ proves \eqref{eq:subbotin-below-first-order}.
\end{proof}

\begin{remark}[The transition at $\beta=1$ and $\beta=2$]\label{rem:subbotin-thresholds}
Fix $\lambda>0$.  The maximizer parameter and the sharp solution factor from Proposition~\ref{prop:subbotin-below-two} satisfy
\[
 u_\beta\longrightarrow0\quad(\beta\uparrow2),
 \qquad
 u_\beta\longrightarrow\infty\quad(\beta\downarrow1),
\]
and therefore
\[
 x_{\beta,\lambda}\longrightarrow0\quad(\beta\uparrow2),
 \qquad
 x_{\beta,\lambda}\longrightarrow\infty\quad(\beta\downarrow1),
\]
while
\[
 C_{\mathrm K}(\beta,\lambda)
 \longrightarrow\frac{\sqrt\pi}{4\sqrt\lambda}
 \quad(\beta\uparrow2),
 \qquad
 C_{\mathrm K}(\beta,\lambda)
 \longrightarrow\frac1\lambda
 \quad(\beta\downarrow1).
\]

For the first endpoint, write $a=1/\beta$ and let $u_0(a)$ be defined by \eqref{eq:R-u0}.  Proposition~\ref{prop:subbotin-below-two} gives $0<u_\beta<u_0(a)$.  For every $\eta>0$, one has $R_{1/2}(\eta)<1/2$; continuity in $a$ and the strict decrease of $R_a$ therefore give $u_0(a)<\eta$ for $a>1/2$ sufficiently close to $1/2$.  Hence $u_\beta\to0$, and \eqref{eq:subbotin-below-factor} yields the Gaussian endpoint value.

For $a\uparrow1$, on the other hand,
\[
 D_a(u)\longrightarrow
 \frac{2(1-e^{-u})}{2-e^{-u}}<1
\]
for every fixed $u>0$.  To make this uniform near the origin, choose $a_0\in(1/2,1)$.  For $a\in[a_0,1]$ and small $u$,
\[
 P_a(u)\le C u^a,\qquad P_a'(u)\le C u^{a-1},
\]
and hence $D_a(u)\le C u^{2a_0-1}$.  After fixing $\delta>0$ so that this is $<1$ on $(0,\delta]$, the convergence above is uniform on each $[\delta,M]$.  Since $D_a(u_\beta)=1$, it follows that $u_\beta\to\infty$.

Put $\epsilon=1-a$.  For $a$ close to $1$, $\gamma(a,1)$ is bounded below by a positive constant, and for $u\ge1$,
\[
 R_a(u)=\frac{u^ae^{-u}}{\gamma(a,u)}\le C u e^{-u}.
\]
Thus $R_a(1/\epsilon)<\epsilon$ for all sufficiently small $\epsilon$.  Since $R_a$ is strictly decreasing and $R_a(u_0)=\epsilon$, we obtain $u_\beta<u_0(a)<1/\epsilon$, so $\epsilon\log u_\beta\to0$.  Moreover $P_a(u_\beta)\to1$, since the upper incomplete-gamma tail is uniformly exponentially small for $a$ near $1$.  Formula~\eqref{eq:subbotin-below-factor-alt} then gives the Laplace endpoint value $1/\lambda$.

At $\beta=1$ itself,
\[
 p_{1,\lambda}(x)=\frac\lambda2e^{-\lambda|x|},
 \qquad
 H_{1,\lambda}(x)=\frac1\lambda\left(1-\frac12e^{-\lambda x}\right),
 \qquad x\ge0,
\]
so the supremum $1/\lambda$ is approached only in the tail.  Thus the zeroth-order envelope exhibits the simple transition
\begin{align*}
 \beta=1 &: \quad \text{tail supremum},\\
 1<\beta<2 &: \quad \text{two off-center maxima},\\
 \beta\ge2 &: \quad \text{maximum at the origin}.
\end{align*}
\end{remark}

\section*{Declaration on AI-assisted preparation}
OpenAI's ChatGPT was used particularly intensively and proved exceptionally efficient in supporting systematic source and novelty searches, including the identification of potentially relevant literature, the tracing of related results, and the cross-checking of priority claims; minor editorial assistance was also provided.  All mathematical and priority conclusions were independently verified by the author.

\raggedbottom

\end{document}